\documentclass[11pt,oneside,reqno]{amsart}
\usepackage{amsfonts}
\usepackage{amsmath}
\usepackage{amssymb}
\usepackage{graphicx}
\usepackage{subcaption}
\usepackage{epstopdf} % with/without is the same
\usepackage[T1]{fontenc}
\usepackage{setspace}
\usepackage{version}
\usepackage{bm}
\usepackage{enumerate}
\usepackage{epsfig}

\usepackage{setspace}
\usepackage{caption}
\DeclareCaptionLabelFormat{cont}
\usepackage
\theoremstyle{plain}
\newtheorem{theorem}{Theorem}[section]

\newtheorem{lemma}[theorem]{Lemma}
\newtheorem{proposition}[theorem]{Proposition}
\newtheorem{definition}[theorem]{Definition}

\newtheorem{remark}[theorem]{Remark}
\newtheorem{consequence}[theorem]{Consequence}
\newtheorem*{notation*}{Notation}
\numberwithin{equation}{section}
\excludeversion{comment1}

\begin{document}
\title{Manifold Data Imputation}
\author{{\bf David Levin}\\ \\School of Mathematical Sciences, Tel-Aviv University, Israel}
%\maketitle
%\footnotetext {David Levin, School of Mathematical Sciences, Tel-Aviv University, Israel}
\begin{abstract}
We consider the problem of reconstructing missing data on a smooth manifold from
incomplete and nonuniform samples. While classical methods for manifold approximation
typically assume quasi-uniform data, their performance deteriorates significantly in the presence of
large gaps or holes. We propose a unified framework for manifold data imputation that
reduces the problem to function reconstruction on locally defined tangent spaces.

The approach combines two complementary strategies. The first is a Fourier-based method
that determines missing values by prescribing a decay rate of the discrete Fourier
coefficients, thereby enforcing high-order smoothness through a global spectral
criterion. The second is a local variational method based on minimizing high-order
central differences, leading to sparse least-squares systems with favorable stability
and conditioning properties. We establish a discrete inverse estimate linking decay of
Fourier coefficients to uniform bounds on high-order divided differences, providing a
theoretical foundation for the spectral approach. For the variational method, we analyze
existence, uniqueness, and scaling behavior, showing that conditioning depends primarily
on the geometry of the missing region.

These functional reconstruction techniques are integrated with a moving least-squares
projection framework to yield a practical algorithm for manifold completion. Numerical
experiments, including reconstruction on surfaces with significant missing regions,
demonstrate accurate and stable recovery without requiring a global parameterization.
The proposed framework provides a flexible and effective approach to manifold data
imputation in challenging settings with incomplete data.
\end{abstract}
\maketitle

\section{Introduction}

Reconstruction of incomplete data on manifolds arises in a wide range of applications,
including geometric modeling, data analysis, and scientific computing. Classical
approaches to scattered data approximation, such as radial basis function methods
\cite{Buhmann2003,Wendland2004}, provide a general framework for approximating
functions from irregular samples. These ideas have been extended to functions defined
on manifolds, where kernels restricted to embedded submanifolds yield approximation
schemes with Sobolev-type error estimates \cite{FuselierWright2012}. Another important
class of methods is based on local approximation. Moving least-squares constructions
have been used to reconstruct manifolds and to approximate functions defined on them
from scattered samples \cite{LipmanLevin2010,SoberLevin2016,SoberAizenbudLevin2017}.
These approaches rely on local coordinate representations, typically obtained from
tangent space approximations, and achieve high-order accuracy under suitable sampling
assumptions.

A complementary direction exploits spectral representations associated with the
geometry of the underlying manifold. Methods based on graph Laplacians or diffusion
operators provide low-frequency representations of functions on manifolds and have
been widely used in manifold learning and data analysis \cite{BelkinNiyogi2003,
CoifmanLafon2006}. Related ideas have also been applied to the approximation of
manifold-valued data and to manifold reconstruction from scattered point clouds
\cite{GrohsSprecherYu2017,FaigenbaumGolovinLevin2020}. Despite these developments,
most existing methods assume a sufficiently dense or quasi-uniform sampling of the
manifold. The reconstruction problem becomes considerably more difficult when the data
are highly nonuniform or when substantial regions of the manifold contain missing
samples.

In this work, we address the problem of \emph{manifold data imputation in the presence
of missing regions (holes)}. The key observation is that, although manifold data lack an intrinsic
grid structure, they can be locally mapped to function data on tangent spaces. This
enables the use of Fourier-based extension and imputation techniques, originally
developed for functions on uniform grids, in a geometric setting.

The functional component of the method is based on prescribing a decay rate for the
discrete Fourier coefficients of the reconstructed data. Such decay conditions reflect
the smoothness of the underlying function and can be translated into bounds on discrete
derivatives. By formulating the extension problem as a weighted least-squares
minimization in the Fourier domain, one obtains reconstructions whose spectral behavior
is controlled in a stable manner. In particular, decay assumptions on the Fourier
coefficients imply uniform bounds on high-order mixed differences, thereby providing a
discrete analogue of classical smoothness results. Related spectral viewpoints can be
found, for example, in \cite{Trefethen2000}.

As a complementary strategy, we consider a local variational method based on minimizing
high-order central differences. This approach enforces smoothness directly in the
physical domain and leads to sparse least-squares systems with a clear geometric
structure. The resulting reconstruction is stable, admits a unique solution under mild
conditions, and exhibits a scaling behavior in which the conditioning is governed
primarily by the geometry of the hole rather than by the mesh size. Moreover, the
effect of the missing region on the discrete derivatives can be quantitatively
controlled, yielding bounded contributions even in the vicinity of the hole. Related
variational ideas appear in PDE-based inpainting methods \cite{Hoeltgen2017}.

To extend these ideas to manifolds, we employ a projection-based reconstruction
framework based on moving least-squares \cite{Meshindependent}, \cite{SoberLevin2016}. By mapping local neighborhoods of the
manifold onto affine tangent spaces, applying the functional imputation methods in
these local coordinates, and lifting the reconstructed data back to the ambient space,
we obtain a practical algorithm for manifold completion. Detection of missing regions
may be carried out using geometric criteria as in \cite{Mindthegap}.

The resulting method is capable of reconstructing missing regions on manifolds from
irregular and incomplete data, without requiring a global parameterization. Numerical
experiments, including reconstruction on surfaces with significant gaps, demonstrate
that the method produces accurate and smooth completions while maintaining stability
in the presence of noise. The combination of spectral and local variational techniques
provides a flexible framework for addressing a broad class of data imputation problems
on manifolds.

\section{\bf Periodic function extension by prescribed decay rate}

Let the domain $D$ be contained
in a box $B \subset \mathbb{R}^d$,
and let $G$ be a uniform grid on $B$.
Assume that the values of a function $f$
are given at the grid points
$X=D \cap G$.
The goal is to extend the function
by determining values
$V_Y=\{v_y\}_{y\in Y}$ on the set
$Y=G\setminus X$
so that the resulting extended function
is $C^M$ and periodic.

The quality of an extension algorithm
can be assessed by examining
the decay rate of the Fourier coefficients
of the extended function.
In \cite{GL10}, it is proposed to derive the extension
by prescribing a desired decay rate
for these coefficients.

Let $B=[0,2\pi]^d$, and let $G$ be a uniform grid of $N^d$ points on $B$.
Consider a $d$-variate $(2\pi)^d$-periodic function $f$ on $\mathbb{R}^d$. We index the points in $G$ as
$$\mathbf{N}=\{\mathbf{n}=(n_1,...,n_d)\ |\  0\le n_j \le N-1,\ \ \ j=1,...,d\}.$$
Consider the discrete Fourier transform defined by the values of $f$ on the grid points,
$$V=\{f(2\mathbf{n}\pi/N)\ |\ \mathbf{n}\in \mathbf{N}\},$$
and let $c_{\mathbf{k}}$ denote the Fourier coefficient with frequency $(k_1,...,k_d)\in \mathbf{N}$:

\begin{equation}\label{c_k}
c_{\mathbf{k}}=\sum_{n\in\mathbf{N}}f(2\mathbf{n}\pi/N)e^{-\frac{2\pi i}{N}\mathbf{k}\cdot\mathbf{n}}.
\end{equation}
We use the non-normalized discrete Fourier transform,
with normalization factor $N^{-d}$ in the inverse transform.
We observe that the coefficients $\{c_{\mathbf{k}}\}$ depend linearly
on the unknown extension values $V_Y$. The algorithm in~\cite{GL10}
determines the extension values $V_Y$
by minimizing the weighted functional
\begin{equation}\label{wfunctional}
C(V_Y)=\sum_{g\in G} |c_{\mathbf{k}}|^2 w_{\mathbf{k}},
\end{equation}
where the weights $w_{\mathbf{k}}$ are selected
to enforce a prescribed decay rate
of the Fourier coefficients.

To motivate the choice of a proper decay rate,
we briefly recall the following result
from~\cite{GL10}.

\begin{lemma}\label{decay}
Let $f$ be a $(2\pi)^d$-periodic function on $\mathbb{R}^d$,
and let
\[
f_n := f\!\left(\frac{2\pi}{N}n\right),
\qquad n \in \mathcal N,
\]
where
\[
\mathcal N = \{0,\dots,N-1\}^d.
\]
Define the discrete Fourier coefficients
\[
c_k = \sum_{n \in \mathcal N}
f_n \, e^{-\,\frac{2\pi i}{N} k \cdot n},
\qquad k \in \mathcal N.
\]
Assume that the mixed derivative
\[
D^{(M,\dots,M)} f
=
\partial_1^M \cdots \partial_d^M f
\]
exists and is bounded on $\mathbb{R}^d$.
Then, for all $k \in \mathcal N$ such that $k_j \neq 0$
for $j=1,\dots,d$,
\begin{equation}
\label{Bound}
|c_k|
\le
N^d(2\pi/N)^{Md}
\|D^{(M,\dots,M)} f\|_\infty
\prod_{j=1}^d
\bigl|e^{-2\pi i k_j/N}-1\bigr|^{-M}.
\end{equation}
\end{lemma}

In Figure~\ref{Bound2} we plot $\log_{10}$ of the bound in~\eqref{Bound}
for $d=2$, $N=50$, and $M=8$, assuming that
$\|D^{(8,8)} f\|_\infty = 1$.

\begin{figure}
\centering
\includegraphics[width=4.in]{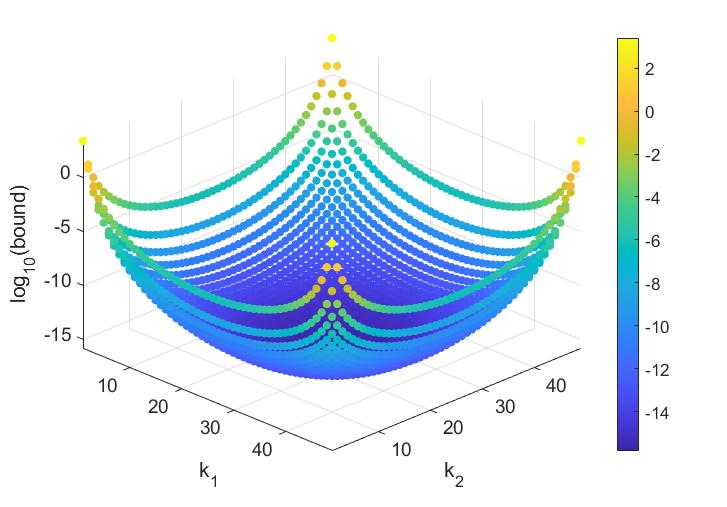}
\hspace{10px}
\caption{Log(Bounds)}
\label{Bound2}
\end{figure}
\medskip

\begin{remark}
In the classical Fourier series setting, an inverse result is well known. 
Let $\{c_k\}_{k\in\mathbb{Z}^d}$ be a sequence of complex coefficients satisfying
\[
\sum_{k\in\mathbb{Z}^d}
\Bigl(\prod_{j=1}^d |k_j|^{\alpha_j}\Bigr)\,|c_k| < \infty
\qquad \text{for every multi-index } \alpha \text{ with } |\alpha|\le M.
\]
Then the Fourier series
\[
f(x)=\sum_{k\in\mathbb{Z}^d} c_k e^{i k\cdot x}
\]
converges absolutely and uniformly together with all its derivatives up to order $M$, and hence defines a $2\pi$-periodic function $f\in C^M(\mathbb{R}^d)$.

This result may be viewed as an inverse theorem for classical Fourier series, asserting that the decay of the Fourier coefficients implies the smoothness of the underlying function. In contrast, for the discrete Fourier transform on a finite $N^d$ grid, no exact inverse statement of this type is available, since only finitely many coefficients are given.
\end{remark}

The following result may be regarded as an inverse theorem for the discrete Fourier transform. It shows that a decay condition on the coefficients, analogous to the classical Fourier setting, yields uniform bounds on mixed divided differences, independent of the number of grid points
$N$. To the best of our knowledge, such an explicit inverse-type estimate
in the discrete setting, with uniform bounds and sharp constants,
is not standard in the literature.

\begin{theorem}
Assume that the coefficients $\{c_k\}_{k\in\mathcal N}$ satisfy
\[
|c_k|
\le
N^d\Bigl(\frac{2\pi}{N}\Bigr)^{Md} C
\prod_{j=1}^d
\bigl|e^{-2\pi i k_j/N}-1\bigr|^{-M},
\qquad k\in\mathcal N,
\]
for some constant $C>0$. Let $h=2\pi/N$. Then the mixed divided
differences of order $(M-2,\dots,M-2)$ satisfy
\[
\|\delta_1^{M-2}\cdots \delta_d^{M-2}f\|_\infty
\le
\Bigl(\frac{\pi^2}{3}\Bigr)^d
\Bigl(1-\frac{1}{N^2}\Bigr)^d C.
\]
\end{theorem}

\begin{proof}
We write
\[
f(x)=\frac{1}{N^d}
\sum_{k\in\mathcal N} c_k e^{ik\cdot x}.
\]
Applying the forward differences gives
\[
\Delta_1^{M-2}\cdots\Delta_d^{M-2}f(x)
=
\frac{1}{N^d}
\sum_{k\in\mathcal N}
\prod_{j=1}^d
\bigl(e^{2\pi i k_j/N}-1\bigr)^{M-2}
c_k\,e^{ik\cdot x}.
\]
Hence
\[
\|\Delta_1^{M-2}\cdots\Delta_d^{M-2}f\|_\infty
\le
\frac{1}{N^d}
\sum_{k\in\mathcal N}
\prod_{j=1}^d
\bigl|e^{2\pi i k_j/N}-1\bigr|^{M-2}
|c_k|.
\]
Substituting the bound on $c_k$, and using
\[
|e^{-2\pi i k_j/N}-1|
=
|e^{2\pi i k_j/N}-1|,
\]
we obtain
\[
\|\Delta_1^{M-2}\cdots\Delta_d^{M-2}f\|_\infty
\le
\Bigl(\frac{2\pi}{N}\Bigr)^{Md}
\sum_{k_1,\dots,k_d=1}^{N-1}
\prod_{j=1}^d
\bigl|e^{2\pi i k_j/N}-1\bigr|^{-2} C.
\]
The sum factorizes, and therefore
\[
\|\Delta_1^{M-2}\cdots\Delta_d^{M-2}f\|_\infty
\le
\Bigl(\frac{2\pi}{N}\Bigr)^{Md}
\left(
\sum_{k=1}^{N-1}
|e^{2\pi i k/N}-1|^{-2}
\right)^d C.
\]
Dividing by $h^{d(M-2)}=(2\pi/N)^{d(M-2)}$ yields
\[
\|\delta_1^{M-2}\cdots \delta_d^{M-2}f\|_\infty
\le
\Bigl(\frac{2\pi}{N}\Bigr)^{2d}
\left(
\sum_{k=1}^{N-1}
|e^{2\pi i k/N}-1|^{-2}
\right)^d C.
\]
Since,  (see, e.g., Gradshteyn and Ryzhik \cite{GR}),
\[
\sum_{k=1}^{N-1}|e^{2\pi i k/N}-1|^{-2}
=
\frac{N^2-1}{12},
\]
we conclude that
\[
\|\delta_1^{M-2}\cdots \delta_d^{M-2}f\|_\infty
\le
\Bigl(\frac{2\pi}{N}\Bigr)^{2d}
\Bigl(\frac{N^2-1}{12}\Bigr)^d C
=
\Bigl(\frac{\pi^2}{3}\Bigr)^d
\Bigl(1-\frac{1}{N^2}\Bigr)^d C.
\]
\end{proof}

\begin{remark}
We emphasize that divided differences  
of orders $M$ and $M-1$  
may be unbounded.  
In the classical Fourier series setting,  
the inverse result leads to  
a loss of one order of smoothness.  
In contrast, for the discrete Fourier transform,  
the loss amounts to two orders.
\end{remark}
\subsection{\bf Properties of the extension by prescribed decay rate}\label{sect1}\hfill

\medskip
Let the domain $D$ be contained in the box $B=[0,2\pi]^d$, and let $X$ be the set of grid points in $D$:
$$X=G\cap D=\{x_\mathbf{n}=2\mathbf{n}\pi/N\in D, |\ \mathbf{n}\in \mathbf{N}\}.$$ Given values of a function $f$ on $X$,
$V_X=\{f(x_\mathbf{n})\ |\ x_\mathbf{n}\in X\}$, we would like to define values $V_Y$ on $Y=G\setminus X$. The extended values $V_Y$, together with the given values
$V_X$, should constitute a smooth periodic set of data values on the grid $G$. Let us denote this combined set of values $V_G$. The way to test
the smoothness and periodicity properties of $V_G$ is by computing
its discrete Fourier transform.
Based on Lemma \ref{decay}, we consider the decay rate of the coefficients of the discrete Fourier transform to be the ultimate measure of the
quality of the extension. The following lemma underpins the extension algorithm presented in \cite{GL10} and also the imputation algorithm suggested in this paper.

\begin{lemma}\label{main}
Assume there exists an extension $V_Y$ such that the coefficients of the discrete Fourier transform of $V_X\cup V_Y$ satisfy
\begin{equation}\label{leek}
|c_{\mathbf{k}}|\le e_{\mathbf{k}},\ \ \mathbf{k}\in\mathbf{N}.
\end{equation}
Consider the extension $V^*_Y$ defined by minimizing
\begin{equation}\label{cost1}
C(V_Y)=\sum_{\mathbf{k}\in\mathbf{N}}|c_{\mathbf{k}}|^2w_{\mathbf{k}},
\end{equation}
over all possible extensions $V_Y$, where
\begin{equation}\label{wk1}
w_{\mathbf{k}}=\frac{1}{e^2_{\mathbf{k}}},\ \ \mathbf{k}\in\mathbf{N},
\end{equation}
then the Fourier coefficients $\{c^*_{\mathbf{k}}\}$ of $V_X\cup V^*_Y$ satisfy
\begin{equation}\label{leek2}
|c^*_{\mathbf{k}}|\le N^{\frac{d}{2}}e_{\mathbf{k}},\ \ \mathbf{k}\in\mathbf{N}.
\end{equation}

\end{lemma}
\begin{proof}
By assumption, there exists an extension $V_Y$ such that
$$C(V_Y)\le \sum_{\mathbf{k}\in\mathbf{N}}e_{\mathbf{k}}^2w_{\mathbf{k}}\le N^d.$$
If for some $\mathbf{k}\in\mathbf{N}$
$$|c^*_{\mathbf{k}}|> N^{\frac{d}{2}}e_{\mathbf{k}},$$
then
$$C(V^*_Y)\ge |c^*_{\mathbf{k}}|^2w_{\mathbf{k}}>N^d,$$
in contradiction with the optimality of $V^*_Y$.
\end{proof}

\begin{consequence}\label{consequence2}{\bf Choosing the weights $\{w_{\mathbf{k}}\}$ in (\ref{cost1}).}

\medskip
Combining the observations in Lemma \ref{decay} and in Lemma \ref{main}, if we look for a $C^M$ periodic extension, we should choose the
weights $\{w_{\mathbf{k}}\}$ in (\ref{cost1}) as:
\begin{equation}\label{wk2}
w_{\mathbf{k}}=\frac{1}{r^2_{\mathbf{k}}},\ \ \mathbf{k}\in\mathbf{N}.
\end{equation}
where
\begin{equation}\label{rk2}
r_{\mathbf{k}}=N^d(2\pi/N)^{Md}
\prod_{j=1}^d
\bigl|e^{-2\pi i k_j/N}-1\bigr|^{-M} ,\ \ \mathbf{k}=(k_1,...,k_d)\in\mathbf{N},
\end{equation}
\end{consequence}

\section{\bf From extension to imputation}

In~\cite{GL10}, the domain $D$ is assumed to be simply connected,
i.e., without holes, and the challenge is to extend the data given on $D$
to a larger domain $B$. The extension smoothly continues the data from $D$ into $B\setminus D$
and prescribes boundary values on $\partial B$ to enforce periodicity. For the imputation task, we allow the domain $D$ to contain holes,
or, more generally, to exhibit missing data.
Let $B=[0,2\pi]^d$  
and let $D \subset B$ be a domain.  
Let $H \subset D$ be an open subset  
on which no data are available,  
with $\overline{H} \subset D$.  
We refer to $H$ as a hole in $D$. 
Let $X$ denote the set of grid points in $D$,
\[
X = G \cap D
= \{x_{\mathbf n} = 2\pi \mathbf n / N \in D
\;|\; \mathbf n \in \mathbb{Z}^d\}.
\]
Given the values of a function $f$ on $X$,
\[
V_X = \{ f(x_{\mathbf n}) \;|\; x_{\mathbf n} \in X \},
\]
we seek to define values
\[
V_Y \quad \text{on} \quad Y = G \setminus X.
\]
The unknown values $V_Y$ include both values
outside $D$ and values within the hole $H \subset D$.
Together with the given values $V_X$,
they form a complete set of values on $B$
over the grid $G$.
We denote this combined set by $V_G$.  
The smoothness and periodicity of $V_G$  
are reflected in the decay  
of its discrete Fourier coefficients.  
A rapid decay  
of the high-frequency coefficients indicates the smoothness 
of the hole imputation.

\subsection{Fourier-based imputation with hyperbolic corner weighting}\hfill

\medskip
As reported in \cite{GL10},  
applying the periodic extension algorithm  
with weights chosen as in  
\eqref{wk2} and \eqref{rk2}  
leads to a highly ill-conditioned  
system of equations. This is highly undesirable,  
particularly in the presence  
of noisy data.

A better-conditioned  
and effective approach,  
proposed in \cite{GL10},  
uses square corner weighting.  
Here we further improve it  
by applying another type of corner weighting derived through the bound $\eqref{rk2}$,
as described below for the two-dimensional case.

Instead of the weights defined by \eqref{wk2},\eqref{rk2}, we use a simple frequency-domain weighting.
Let $W=\{w_{k_1,k_2}\}_{k_1=0,\ldots,N-1,\; k_2=0,\ldots,N-1}$ be the array
of weights in the Fourier domain.
We choose a bound $C$ and set
\[
w_{k_1,k_2}=1
\]
if $|r_{k_1,k_2}|< C$, and $W\equiv 0$ elsewhere.
This choice yields a sharp truncation pattern in the discrete frequency
grid. The set of indices removed by the mask 
has a structure reminiscent of a hyperbolic cross,
reflecting the anisotropic decay of the bound.  An example of this hyperbolic-type selection  
is shown in Figure~\ref{Hyp0001},  
for $N=50$ and $C=0.001$.  
The weights are set to zero  
at the selected points  
and to one elsewhere.

\begin{figure}
\centering
\includegraphics[width=3.5in]{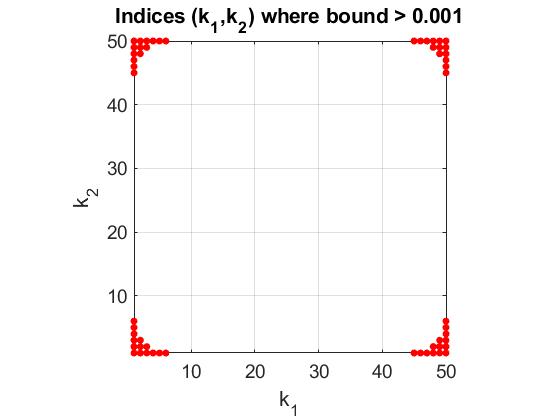}
\hspace{10px}
\caption{Hyperbolic weights}
\label{Hyp0001}
\end{figure}
\medskip

Since the minimization functional has the form
\[
\sum_{\mathbf{k}} w_{\mathbf{k}}\,|c_{\mathbf{k}}|^2,\ \ \ \mathbf{k}=\{k_1,k_2\}
\]
only those Fourier coefficients for which $w_{\mathbf{k}}>0$ are penalized.
The modes located in the triangular corner regions, where
$w_{\mathbf{k}}=0$, are therefore left unpenalized, while all remaining modes
are penalized uniformly.
In the standard unshifted discrete Fourier indexing, these corner regions
correspond to the low-frequency part of the spectrum (including the
constant mode and nearby modes).
Consequently, this choice leaves the low-frequency components essentially
unconstrained and suppresses the complementary part of the spectrum,
which corresponds to higher-frequency content.

Using the expression \eqref{c_k}  
for the Fourier coefficients,  
the minimization problem  
leads to a linear system  
\[
A V_Y = b,
\]
where $V_Y$ denotes  
the unknown values on $Y$.  
The completed data  
are then inserted  
into the full  
$N\times N$ grid.
To assess the numerical behavior 
of the method,  
we report the condition number  
of the matrix $A$ 
and display the reconstructed surface  
over the hole and over the entire grid.  
In addition, we plot  
the absolute values  
of the discrete Fourier coefficients  
of the completed data 
to visualize  
the spectral effect  
of the chosen weighting. 

We consider a Fourier-based imputation problem $B=[0,2\pi]^2$,  
discretized on a uniform $50\times 50$ grid $G$. 
We consider a domain $D$ to be an annulus centered at $(\pi,\pi)$, with an outer radius $\pi/2$ and a hole of radius $0.8$. 
The data are prescribed on $X \subset D$,  
with missing values on $Y = G \setminus X$,  
including both the hole and the exterior of $D$.

The underlying test function is
\[
f(x,y)=\frac{1}{2.5+\sin(x+1.2)+\cos(y)},
\]
and noisy data are obtained  
by adding an independent uniform noise 
of amplitude $\varepsilon=0.1$  
to the data on $X$.

The reconstruction is formulated  
as the minimization of a weighted quadratic functional  
in the Fourier domain,  
leading to a linear system for the unknown values on $Y$. As described above, the choice of weights is derived from the theoretical bound  
and modified by truncation in selected frequency regions. It 
plays a central role in controlling the conditioning  
of the resulting system. The weights follow  
the pattern shown in Figure~\ref{Hyp0001},  
vanishing at the selected red points  
and equal to one elsewhere. The set $X$ of known data points  
contains $360$ points,  
while $2140$ values  
remain to be determined on the set $Y$. For the corresponding $2160\times 2160$ matrix $A$,  
we obtain  
\[
\operatorname{cond}(A)=1.0\times 10^{8},
\]
indicating moderate ill-conditioning.

The performance of the method  
is evaluated by the reconstructed surface  
shown in Figure~\ref{Surface}(A),  
and by the decay  
of the discrete Fourier coefficients  
shown in Figure~\ref{Surface}(B). We observe in panel (B)  
that the significant Fourier coefficients  
correspond to low-frequency components.  
The high-frequency components  
are not identically zero,  
but their magnitude  
is below $4\times 10^{-4}$.

The reconstruction of the missing data 
in the hole is presented  
in Figure~\ref{Reconstruction2}.  
Panel (A) shows the reconstruction  
together with the noisy data  
on the annulus,  
whereas panel (B) shows it  
together with the exact data  
on the annulus.
The reconstruction error  
on the hole is $0.22$,  
indicating stable recovery  
despite the noise level $0.1$.

\begin{figure}[ht]
\centering

\begin{subfigure}{0.48\textwidth}
    \centering
\includegraphics[width=1.\linewidth]{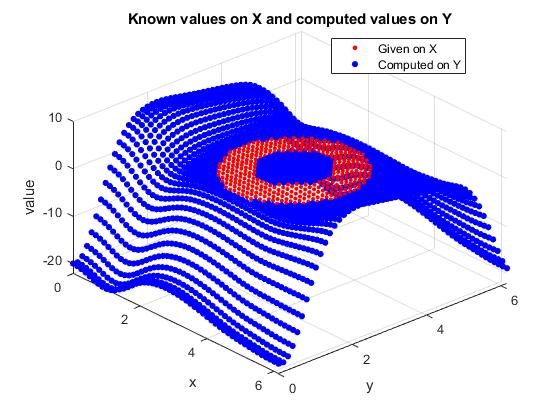}
    \caption{Known and reconstructed data}

\end{subfigure}
%\hfill
\begin{subfigure}{0.48\textwidth}
    \centering
    \includegraphics[width=1.\linewidth]{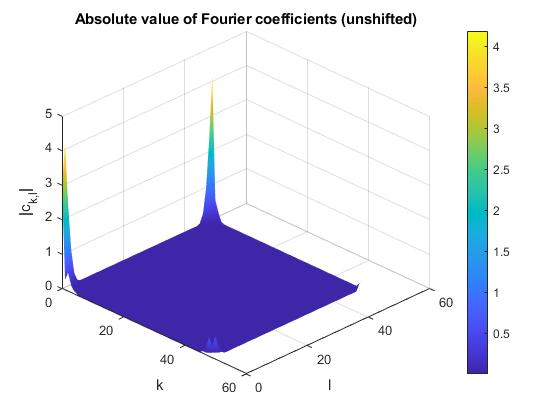}
    \caption{Fourier coefficients}
\end{subfigure}

\caption{}
\label{Surface}
\end{figure}

\begin{figure}[ht]
\centering

\begin{subfigure}{0.48\textwidth}
    \centering
\includegraphics[width=1.\linewidth]{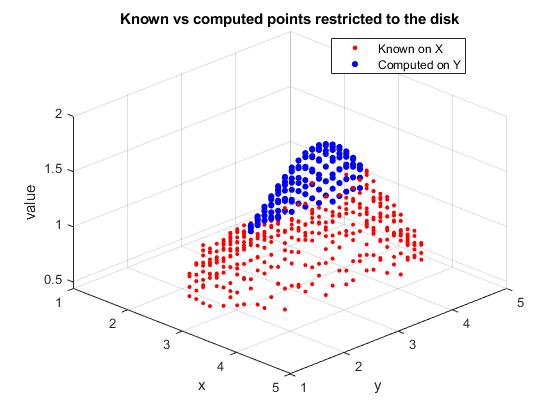}
\caption{The reconstructed data in the hole, displayed together with the noisy data on the annulus}

\end{subfigure}
%\hfill
\begin{subfigure}{0.48\textwidth}
    \centering
    \includegraphics[width=1.\linewidth]{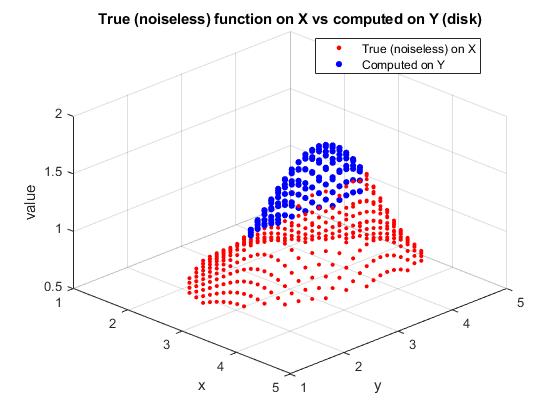}
    \caption{The reconstructed data in the hole, displayed together with the exact data on the annulus}
\end{subfigure}

\caption{}
\label{Reconstruction2}
\end{figure}

%%%%%%%%%%%%%%%%

\section{Hole filling via minimization of high-order differences in $\mathbb{R}^d$}
\label{HoleFA}

Consider the problem of recovering missing data
on a subset of a uniform grid in $\mathbb{R}^d$.
Let $G\subset\mathbb{R}^d$ be a Cartesian grid
with mesh size $h$,
and let $\Omega\subset G$ denote the set of grid points
where data are prescribed.

Assume that $\Omega$ contains a hole,
i.e., a subset $H\subset G$ where the values are unknown.
We denote by $X=\Omega\setminus H$ the set of known points,
and by $Y=H$ the set of unknown points to be recovered. Concerning the size of the hole,
we distinguish between two cases:
(I) The hole is contained in a ball
of diameter $q h$, where $q$ is fixed,
and (II) the hole is contained in a ball
of diameter $D$,
independent of $h$.

\medskip

\noindent
Let $f:X\to\mathbb{R}$ be the given data.
We seek an extension $u:G\to\mathbb{R}$ such that
\[
u|_X = f,
\]
and such that $u$ is as smooth as possible
in a discrete sense.

\medskip

\noindent
To this end, we define the cost functional
\begin{equation}
    \label{COST}
\mathrm{J}(u)
=
\sum_{i\in G_k}
\sum_{j=1}^d
\bigl(\Delta_j^{2k} u(i)\bigr)^2,
\end{equation}
where $\Delta_j^{2k}$ denotes the $2k$-th order
central difference operator in the $j$-th coordinate direction,
and $G_k\subset G$ is the subset of grid points
whose distance from the boundary of the computational domain
is at least $k$ grid steps.

\medskip

\noindent
The reconstruction problem is therefore formulated as:
find $u$ such that
\begin{equation}
\label{uXf}
u|_X = f,
\end{equation}
and
\begin{equation}
\label{argminCOST}
u = \arg\min \mathrm{J}(u).
\end{equation}

\medskip

\noindent
This leads to an overdetermined linear least-squares system
for the unknown values $\{u(y)\}_{y\in Y}$.
The resulting system is sparse and structured,
since each equation involves only a local stencil
of size $2k+1$ in each coordinate direction.

\medskip

\noindent
In practice, the minimization is performed
on a local rectangular domain $B\subset G$
that contains the hole $H$ and is enlarged
by $k+1$ layers of grid points in each coordinate direction.
This ensures that all difference stencils are fully supported.

\medskip

\noindent
The method can be interpreted as constructing
a discrete extension,
in which high-order differences are minimized,
leading to a smooth continuation of the data
into the missing region.

\begin{proposition}[Bounds on high-order differences
involving unknown data]
Let $u_h$ be the minimizer of
\[
\mathrm{J}(u)
=
\sum_{i\in G_k}\sum_{j=1}^d
\bigl(\Delta_j^{2k}u(i)\bigr)^2
\]
subject to $u_h|_X = f|_X$,
where $X$ is the set of known grid points,
and let $H$ denote the hole.

Assume that $f\in C^M(B)$,
with $M\ge 2k$.
Let $S_h$ be the set of all $(i,j)$
such that the stencil of $\Delta_j^{2k}u(i)$
contains at least one unknown point from $H$.

Then the following hold.

\begin{enumerate}
\item[(I)]
If the hole is contained in a ball
of diameter $qh$, then
\[
|S_h|\le C(q,d,k),
\]
and
\[
\bigl|\Delta_j^{2k}u_h(i)\bigr|
\le C\,h^{2k},
\qquad (i,j)\in S_h.
\]

\item[(II)]
If the hole is contained in a ball
of diameter $D$, independent of $h$,
then
\[
|S_h|\le C(D,d,k)\,h^{-d},
\]
and
\[
\bigl|\Delta_j^{2k}u_h(i)\bigr|
\le C\,h^{2k-d/2},
\qquad (i,j)\in S_h.
\]
\end{enumerate}

The constant $C$ depends on
$D$, $d$, $k$,
and $\|f\|_{C^{2k}(B)}$,
but is independent of $h$.
\end{proposition}

\begin{proof}
Since $u_h$ minimizes $\mathrm{J}$,
and the exact data $f$ define an admissible extension,
we have
\[
\sum_{i,j}
\bigl(\Delta_j^{2k}u_h(i)\bigr)^2
\le
\sum_{i,j}
\bigl(\Delta_j^{2k}f(i)\bigr)^2.
\]
In particular,
\[
\sum_{(i,j)\in S_h}
\bigl(\Delta_j^{2k}u_h(i)\bigr)^2
\le
\sum_{(i,j)\in S_h}
\bigl(\Delta_j^{2k}f(i)\bigr)^2.
\]

Since $f\in C^{2k}(B)$,
\[
\bigl|\Delta_j^{2k}f(i)\bigr|
\le C\,h^{2k},
\]
and hence
\[
\bigl(\Delta_j^{2k}f(i)\bigr)^2
\le C\,h^{4k}.
\]

In case (I),
the number of affected stencils is bounded,
so
\[
\sum_{(i,j)\in S_h}
\bigl(\Delta_j^{2k}f(i)\bigr)^2
\le C\,h^{4k}.
\]
In case (II),
the number of affected stencils is $O(h^{-d})$,
so
\[
\sum_{(i,j)\in S_h}
\bigl(\Delta_j^{2k}f(i)\bigr)^2
\le C\,h^{4k-d}.
\]

Since each term is bounded by the total sum,
the stated pointwise bounds follow.
\end{proof}

\begin{theorem}[Existence, uniqueness, and scaling of the local
high-order difference reconstruction]
Let $B\subset h\mathbb{Z}^d$ be a rectangular grid patch,
let $X\subset B$ be the set of known grid points,
and let $Y\subset B$ be the set of unknown points.
Assume that the hole $Y$ is strictly contained in $B$,
with at least $k$ grid layers between $Y$
and the boundary of $B$,
so that every centered stencil of width $2k+1$
used in the definition of the functional is fully supported.

Define the quadratic functional
\[
J(u)
=
\sum_{\nu\in B_k}\sum_{j=1}^d
\bigl(\Delta_j^{2k}u(\nu)\bigr)^2,
\]
over all grid functions $u$ on $B$
satisfying the interpolation constraints
\[
u|_X=f.
\]
Let $v\in\mathbb{R}^{|Y|}$ denote the vector of unknown values,
and let
\[
J(v)=\|Av-b\|_2^2
\]
be the corresponding least-squares formulation.

Assume that the only grid function $w$ supported on $Y$
for which
\[
\Delta_j^{2k} w(\nu)=0,
\qquad \nu\in B_k,\quad j=1,\dots,d,
\]
is the zero function.
Then the following hold:

\begin{enumerate}
\item[(i)]
The minimization problem admits a unique solution.

\item[(ii)]
The normal matrix $A^TA$ is symmetric and positive definite,
and the minimizer is characterized by the normal equations
\[
A^TA\,v=A^Tb.
\]
\end{enumerate}
\end{theorem}

\begin{proof}
After inserting the known values on $X$,
the functional $J$ becomes a quadratic polynomial
in the unknown vector $v\in\mathbb{R}^{|Y|}$,
of the form
\[
J(v)=\|Av-b\|_2^2.
\]
Hence a minimizer always exists.
It is unique if and only if the homogeneous problem
$Av=0$ has only the trivial solution,
which is exactly the stated assumption.
Under this condition, $A$ has full column rank,
so $A^TA$ is symmetric positive definite,
and the minimizer satisfies the normal equations.
\end{proof}

The proposed method can be interpreted as enforcing a high degree of
discrete smoothness in the reconstruction.
For $k=1$, the functional reduces to a discrete Laplacian-based
smoothing, while for larger values of $k$ it penalizes higher-order
variations.
In this sense, the method is closely related to polyharmonic and
thin-plate spline interpolation, and provides a flexible mechanism for
controlling the smoothness of the reconstructed surface.

\paragraph{\bf Comparison with the Fourier method.}\hfill

\medskip

The Fourier-based completion method is global in nature.  
It determines the missing values by enforcing a prescribed decay  
of the discrete Fourier coefficients of a periodic extension,  
thereby controlling smoothness through a spectral criterion.  
In view of the inverse estimate in Section~2,  
such decay conditions imply uniform bounds  
on high-order mixed divided differences,  
providing a global mechanism for smoothness control.

In contrast, the present method is purely local.  
It reconstructs the data inside the hole  
by minimizing high-order central differences  
in a neighborhood of the missing region.  
As shown in Proposition~4.1,  
the contribution of the unknown values  
to the high-order differences remains controlled,  
with bounds that depend primarily  
on the geometry of the hole.  
Moreover, Theorem~4.2 guarantees  
existence, uniqueness, and favorable conditioning  
of the resulting least-squares system.

A key distinction between the two approaches  
lies in their numerical behavior.  
The Fourier method leads to dense systems  
whose conditioning depends critically  
on the choice of spectral weights.  
With hyperbolic-type weighting,  
the resulting systems are only moderately ill-conditioned  
and can accommodate data with relatively high noise levels.  
In contrast, the variational formulation  
produces sparse and well-structured systems,  
with conditioning governed primarily  
by the geometry of the hole,  
and exhibits stable behavior  
for low noise levels.

\subsection{ Numerical experiment - 2D}\hfill

Let
\[
B=[0,2\pi]^2,
\qquad
G=\{(x_m,y_n)\}_{m,n=0,\ldots,M-1},
\]
be a uniform grid with
\[
x_m=\frac{2\pi m}{M},
\qquad
y_n=\frac{2\pi n}{M},
\]
and $M=40$.
We assume that the data are given on a subset $X\subset G$, and are
missing on a subset $H\subset G$, referred to as the hole.
In the present experiment, the set $X$ corresponds to an annular region,
while $H$ is a disk of radius $0.5$ centered at $(\pi,\pi)$.
The known values on $X$  
are sampled from the function
\[
f(x,y)=\frac{1}{2.5+\sin(x+1.2)+\cos(y)},
\]
with added independent uniform noise  
of amplitude $0.01$.

In order to recover the missing values inside the hole, we proceed as
follows.
We first construct a minimal axis-aligned rectangle
\[
R_k=[a,b]\times[c,d]
\]
containing the hole, such that all grid points in $H$ lie at a distance
strictly greater than $k$ grid steps from the boundary of $R_k$.
We then restrict the reconstruction problem to the grid points lying in
$R_k$.
On this restricted domain, the values are known at points in $X\cap R$
and unknown at points in $H$.

Let $f_{i,j}$ denote the values of the function on the grid points in $R$.
We define a cost functional
\[
\mathrm{J}=\sum_{(i,j)\in R_k}
\left[
\bigl(\Delta_x^{2k} f_{i,j}\bigr)^2
+
\bigl(\Delta_y^{2k} f_{i,j}\bigr)^2
\right],
\]
where $\Delta_x^{2k}$ and $\Delta_y^{2k}$ denote the central difference
operators of order $2k$ in the $x$ and $y$ directions, respectively, and
$R_k$ is the subset of indices $(i,j)$ such that the corresponding grid
points are at a distance at least $k$ from the boundary of $R$.

The reconstruction of the missing data is obtained by minimizing the
functional $\mathrm{J}$ with respect to the unknown values at the grid
points in $H$.
This leads to a linear least-squares problem of the form
\[
A u = b,
\]
where $u$ is the vector of unknown values inside the hole.
Each row of the matrix $A$ corresponds to one evaluation of a
high-order difference operator at a grid point in $R_k$, and the
right-hand side $b$ collects the contributions from neighboring points
with known data values.

Figure \ref{Diffrec}(A) shows a pointwise visualization in which
the known values are shown in blue
and the reconstructed values are shown in red.
This representation highlights the consistency
of the recovered values with the surrounding data
and demonstrates that the reconstruction forms
a coherent continuation of the surface.
No spurious oscillations are visible inside the hole,
despite the presence of noise.

Figure \ref{Diffrec}(B) presents an image plot of the reconstructed function
on the local rectangle, providing a top view
that clearly illustrates the smooth transition
across the former hole.
Together, these figures demonstrate  
that the proposed local least-squares approach,  
based on high-order central differences,  
yields accurate and stable hole filling  
for this disk-with-hole configuration.  
The condition number of $A^TA$  
is approximately $1.25\times 10^{4}$,  
indicating favorable conditioning.  
However, the method exhibits sensitivity  
to higher levels of noise.

\begin{figure}[ht]
\centering

\begin{subfigure}{0.48\textwidth}
    \centering
\includegraphics[width=1.\linewidth]{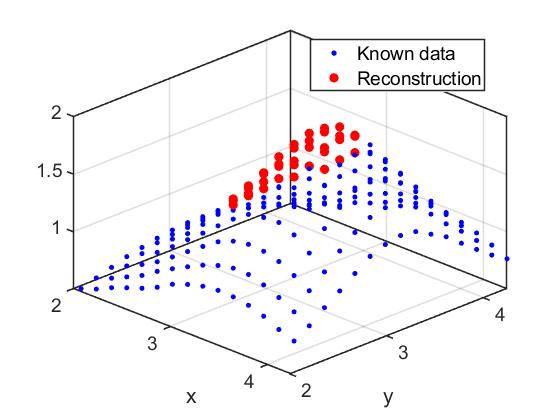}
    \caption{Known and reconstructed data}

\end{subfigure}
%\hfill
\begin{subfigure}{0.48\textwidth}
    \centering
    \includegraphics[width=1.\linewidth]{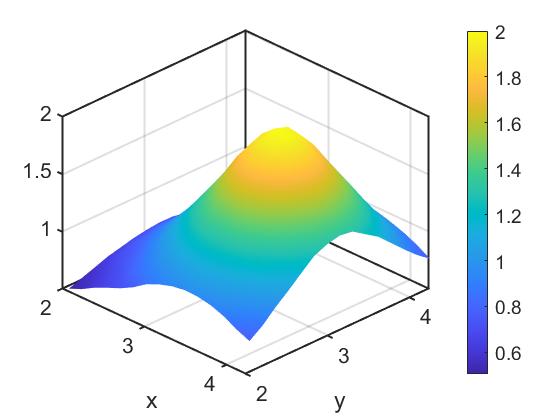}
    \caption{Resulting surface}
\end{subfigure}

\caption{}
\label{Diffrec}
\end{figure}

\subsection{Error estimates}\hfill
\medskip

To estimate the reconstruction error,
it is sufficient to consider
one-dimensional sections of the problem.
Indeed, the functional \eqref{COST}
is separable with respect to the coordinate directions,
and each high-order difference operator
acts along a single coordinate line.
Therefore, the multidimensional error
can be controlled by analyzing
the corresponding one-dimensional problem.

Let $u$ denote the underlying function,
and assume that the data on $X$
are contaminated by independent random errors
of amplitude $\varepsilon$.
Thus, the prescribed values are of the form
\[
f^\varepsilon = u + \eta,
\qquad \|\eta\|_\infty \le \varepsilon.
\]
Let $u_h$ be the reconstruction obtained
by minimizing \eqref{COST}
subject to the interpolation constraints
\[
u_h|_X = f^\varepsilon|_X.
\]

We define the error
\[
e_h := u_h - u,
\]
and introduce the residual
\[
r := A_h e_h,
\]
where $A_h$ is the linear operator
associated with the discrete $2k$-th order differences
on the local reconstruction patch.

The residual $r$ consists of two contributions:
a consistency term associated with the smoothness of $u$,
and a perturbation induced by the noise.
By Proposition~4.1,
the consistency term satisfies
\[
\|r\|_\infty = O(h^{2k})
\quad \text{in case (I),}
\qquad
\|r\|_\infty = O(h^{2k-d/2})
\quad \text{in case (II),}
\]
for exact data.
The action of the high-order difference operator
on the noise contributes an additional term
of order $\varepsilon$.
Thus,
\[
\|r\|_\infty
\le C\bigl(h^{2k} + \varepsilon\bigr)
\quad \text{in case (I),}
\]
and
\[
\|r\|_\infty
\le C\bigl(h^{2k-d/2} + \varepsilon\bigr)
\quad \text{in case (II).}
\]

To estimate the error $e_h$,
we use the representation
\[
e_h = A_h^{-1} r.
\]
Since the reconstruction is performed
on a local rectangular patch,
in which the unknown values are separated
from the boundary,
the operator $A_h$ corresponds
to a high-order difference operator
with homogeneous exterior conditions.
More precisely, after incorporating the prescribed values on $X$, 
the operator $A_h$ acts only on the unknown values in the interior 
of the local patch, with the surrounding known data inducing 
homogeneous exterior conditions for the difference stencils. 
In this form, $A_h$ coincides with a centered $2k$-th order 
difference operator with homogeneous exterior conditions, 
as required in Appendix~A. We identify $A_h$ with the matrix $A$ in Appendix~A, up to scaling 
by the mesh size.
Therefore, the bound in Appendix~A applies
on the local patch.

Let $\ell_h$ denote the number of grid points
across the diameter of the local reconstruction domain in one coordinate direction.
Then
\[
\|A_h^{-1}\|_\infty \le C\, \ell_h^{\,2k}.
\]
Hence,
\[
\|u_h - u\|_\infty
\le C\, \ell_h^{\,2k}\, \|r\|_\infty.
\]

We now distinguish between two regimes.

\medskip
\noindent
\textbf{(I) Small hole.}
Assume that the hole $H$
is contained in a ball
of diameter $qh$,
where $q$ is independent of $h$.
Then the local reconstruction patch
contains only $O(1)$ grid points,
and therefore
\[
\ell_h = O(1).
\]
Consequently,
\[
\|u_h - u\|_\infty
\le C \bigl(h^{2k} + \varepsilon\bigr).
\]
Thus, in the small-hole regime,
the reconstruction achieves
high-order accuracy,
with an additive perturbation
due to the noise.

\medskip
\noindent
\textbf{(II) Large hole.}
Assume that the hole $H$
is contained in a ball
of diameter $D$,
independent of $h$.
Then the local reconstruction patch
has diameter of order one,
and therefore
\[
\ell_h = O(h^{-1}).
\]
Hence,
\[
\|A_h^{-1}\|_\infty \le C h^{-2k},
\]
and
\[
\|u_h - u\|_\infty
\le C h^{-2k}
\bigl(h^{2k-d/2} + \varepsilon\bigr)
= C \bigl(h^{-d/2} + h^{-2k}\varepsilon\bigr).
\]

This estimate shows that,
for holes of fixed physical size,
the reconstruction error may be amplified
as the mesh is refined,
and the effect of noise becomes dominant.
In contrast,
for holes of diameter $O(h)$,
the amplification factor remains bounded,
and the method retains its stability.
\
%%%%%%%%%%%%%%%
\section{Imputation on manifolds}

We now turn to the main objective of this paper,
namely the approximation of manifolds in $\mathbb{R}^n$,
and, in particular, the completion of a manifold
in the presence of missing data.
The approach proposed here combines
the hole-filling algorithms for functional data
developed in the preceding sections
with a projection-based method
for manifold approximation,
which is described below.

\subsection{Manifold moving least-squares projection (MMLS)}\label{MMLS}
\hfill

The moving least-squares projection method introduced in
\cite{SoberLevin2016}
provides a framework for approximating a smooth
$d$–dimensional manifold from scattered data points
in $\mathbb{R}^n$.
The method extends the projection-based approach
for manifold approximation,
generalizing the construction presented in
\cite{Meshindependent}
for surfaces in $\mathbb{R}^3$
to manifolds of arbitrary dimension.
In both works, it is assumed that the data points
are quasi-uniformly distributed on the manifold
and that no significant gaps are present.
The treatment of missing data, and in particular
the completion of manifolds in the presence of holes,
is precisely the issue addressed in the present paper.

Given scattered data points on a smooth manifold $\mathcal{M}$,
and a point $x$ in its neighborhood,
the method constructs a local coordinate system
by fitting a $d$–dimensional affine subspace $H_x$
to the nearby data
via a weighted least-squares procedure.
The weights are typically chosen as a rapidly decaying function
of the distance from $x$, so that nearby data points
have a dominant influence.

Once the local affine space $H_x$ is determined,
the data are locally approximated by a smooth function
defined over $H_x$ via a second least-squares fit.
The projection of $x$ onto the approximated manifold
is then defined as the evaluation of this local model at the origin
of the coordinate system associated with $H_x$.
This yields a projection mapping $Q(x)$ that sends nearby points
onto the reconstructed manifold.

Under suitable sampling and smoothness assumptions,
the MMLS projection is well defined, smooth,
and provides a high-order approximation of the underlying manifold.
In particular, the method achieves approximation orders
comparable to classical moving least-squares schemes,
while being robust to noise and irregular sampling. The principal ingredients 
of the MMLS projection procedure 
are the degree 
of the local approximating polynomial 
and the choice of weight functions 
for computing 
the local affine space $H_x$ 
and the local polynomial.

\subsection{The manifold hole-filling algorithm}\hfill

Given scattered data points on a smooth manifold $\mathcal{M}$,
we address the problem of recovering
regions of the manifold in which the data are missing.

Similar to the notion of mesh size in the above case of functional data, here we talk about a filling distance. 
The filling distance of a point set $P\subset\mathcal{M}$
on a manifold $\mathcal{M}$ is defined by
\[
h_{P,\mathcal{M}}
=
\sup_{x\in\mathcal{M}}
\min_{p_i\in P} d_{\mathcal{M}}(x,p_i),
\]
where $d_{\mathcal{M}}$ denotes the geodesic distance on $\mathcal{M}$. For our case of missing data, we also consider the filling distance with respect to a subset $A$ of the manifold around the missing region in the data.
It measures how well data surrounds the missing region.

\begin{definition}
Let $\mathcal{M}\subset\mathbb{R}^n$ be a smooth manifold,
let $P\subset\mathcal{M}$ be a finite set of points,
and let $A\subset\mathcal{M}$.

The restricted filling distance of $P$ with respect to $A$
is defined by
\[
h_{P,A}
=
\sup_{x\in A}
\min_{p_i\in P}
d_{\mathcal{M}}(x,p_i).
\]
\end{definition}

Let $B=\mathcal{M}\setminus A$ be the missing region of $\mathcal{M}$, and assume that
\begin{equation}\label{OhPA}
diam(B)= O(h_{P,A}).
\end{equation}

Unlike the functional case,
in which hole-filling data is performed on a grid,
the manifold setting does not admit
an underlying grid structure for data completion. Accordingly, the task is not merely data completion,
but the completion of the manifold itself. We proceed to describe the proposed manifold reconstruction algorithm,
which we refer to as the hole-filling algorithm. We illustrate the algorithmic steps 
through the imputation 
of a two-dimensional manifold 
in $\mathbb{R}^3$, 
i.e., a surface, 
as shown in Figure~\ref{Torushole}.

\medskip
{\bf The hole-filling algorithm}

\medskip
\begin{enumerate}
\item Identify the location of a hole $B$
in the data set $P$. 
This may be achieved 
by detecting points $q_j\in P$, 
$j=1,\ldots,J$, 
that lack neighbors 
in certain directions. 
The algorithm in \cite{Mindthegap} 
can be used for this purpose.
In addition, estimate 
the size of the hole, 
i.e., its diameter ${diam}(B)$ and its center $\tilde q=\sum q_j/J$.  In Figure \ref{Torushole}, we see data points on the surface of a torus and a hole in the data.
\item Construct a hyperplane $\Pi$ 
that approximates the tangent space 
to $\mathcal{M}$ 
in a neighborhood of the hole. 
This can be achieved 
by estimating local tangent hyperplanes 
$\{\Pi_j\}$ at sample points $\{q_j\}$, 
and averaging these hyperplanes. 
A portion of the resulting plane 
is shown in green in Figure~\ref{Torushole}. 
Its central part is removed 
to reveal the hole in the data.
\item Let $\tilde{\pi}$ be the orthogonal projection 
of $\tilde{q}$ onto $\Pi$. 
Fix an orthonormal basis of $\Pi$, 
and introduce Cartesian coordinates 
centered at $\tilde{\pi}$. 
Define 
a uniform square mesh 
on $\Pi$ with mesh size $\sim h_{P,A}$. 
The mesh is shown in red in Figure~\ref{Torushole}.

\item Let $R\subset\Pi$ be a cube 
of edge length $2{diam}(B)$, 
centered at $\tilde{\pi}$, 
and denote by $\{x_i\}$ 
the grid points of the mesh in $R$.

\item For each grid point $x_i$, 
we apply the MMLS projection algorithm 
described in Section~\ref{MMLS}. 
The method assumes 
that sufficiently many data points 
are available 
in a neighborhood of the projection 
of $x_i$ onto the manifold.

As seen in Figure~\ref{Torushole}, 
this assumption fails 
for grid points near the hole, 
where the available data are insufficient 
for a reliable projection. 
In regions of $R$ 
at a distance $> h_{P,A}$ from the hole, 
the MMLS projection yields 
a smooth, high-order approximation 
of the manifold. 
The resulting approximation is a $d$-dimensional vector of smooth functions,
$$\overline{\mathbf{P}}(x)=\big (P_1(x), P_2(x),...,P_d(x)\big).$$
In the torus example, 
for each admissible grid point $x_i$ 
on the reference plane, 
the projection yields a point 
$\big(P_1(x_i),P_2(x_i),P_3(x_i)\big)$ 
at a small distance 
from the torus.
The blue points in Figure~\ref{RecTorus} 
result from projecting 
points sampled on a fine mesh 
of the reference plane.
\item For each component function $P_j$, 
its values are available 
on a regular grid, 
excluding grid points 
corresponding to the missing region 
on the manifold. 
Accordingly, for each $P_j$, 
the setting is suitable 
for applying the functional 
hole-filling algorithms 
presented in Sections~2--4.
The red points in Figure~\ref{RecTorus} 
are generated by applying 
the hole-filling algorithm 
of Section~\ref{HoleFA} 
independently to each component function, 
using the blue points as the given data.
\item The mesh structure 
of the reconstructed data 
in the vicinity 
of the missing region 
facilitates a direct approximation 
of the manifold 
by a spline patch.
\end{enumerate}

\medskip
{\bf Summary of the 3D Numerical Example Illustrated in the Figures:}\hfill

\medskip
We are given $\sim 2800$ scattered data points sampled from the surface of a variable-radius torus, with a localized region of missing data. In the MMLS projection step, 
the local coordinate system 
and the local polynomial approximation 
are defined using a Gaussian weight function, 
with bivariate polynomials 
of total degree $5$. This results in the blue points in Figure \ref{RecTorus}. The error analysis in \cite{SoberLevin2016} indicates that the approximation error is of order $O(h_{P.A}^6)$. The observed maximal error 
between the blue reconstructed points 
and the torus surface, 
is approximately 
$2.5\times 10^{-5}$. In the hole-filling algorithm, the procedure described in Section~\ref{HoleFA} 
is applied with $k=3$. The observed error 
between the red reconstructed points 
and the torus surface 
is approximately 
$6\times 10^{-5}$. 
This error reflects 
both the error in the input data 
(the blue points) 
and the error introduced 
by the hole-filling procedure.

\begin{figure}
\centering
\includegraphics[width=5.3 in]{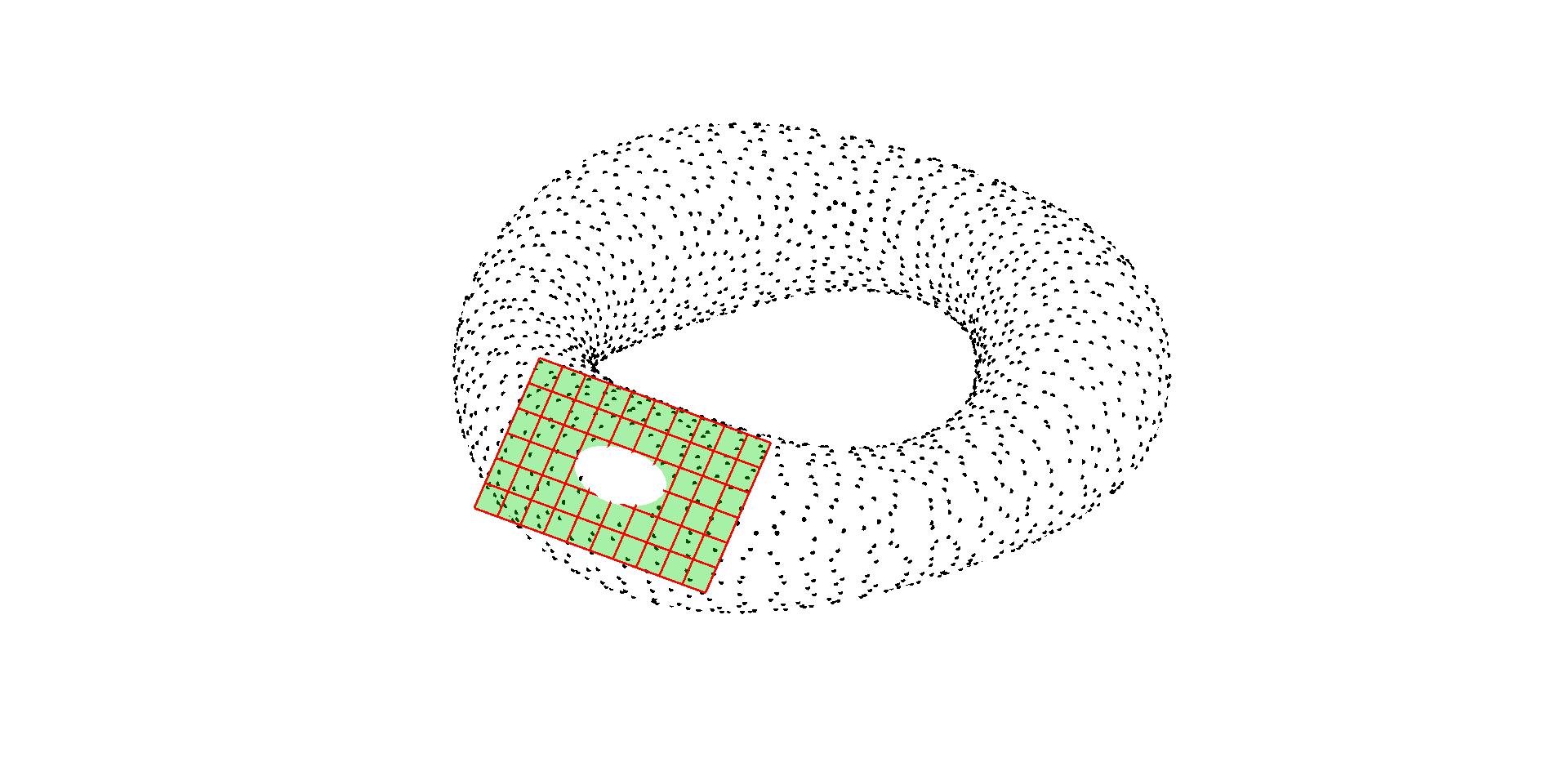}
\caption{Preparing the tangent plane near the hole in the data}
\label{Torushole}
\end{figure}
\medskip

\begin{figure}
\centering
\includegraphics[width=5 in]{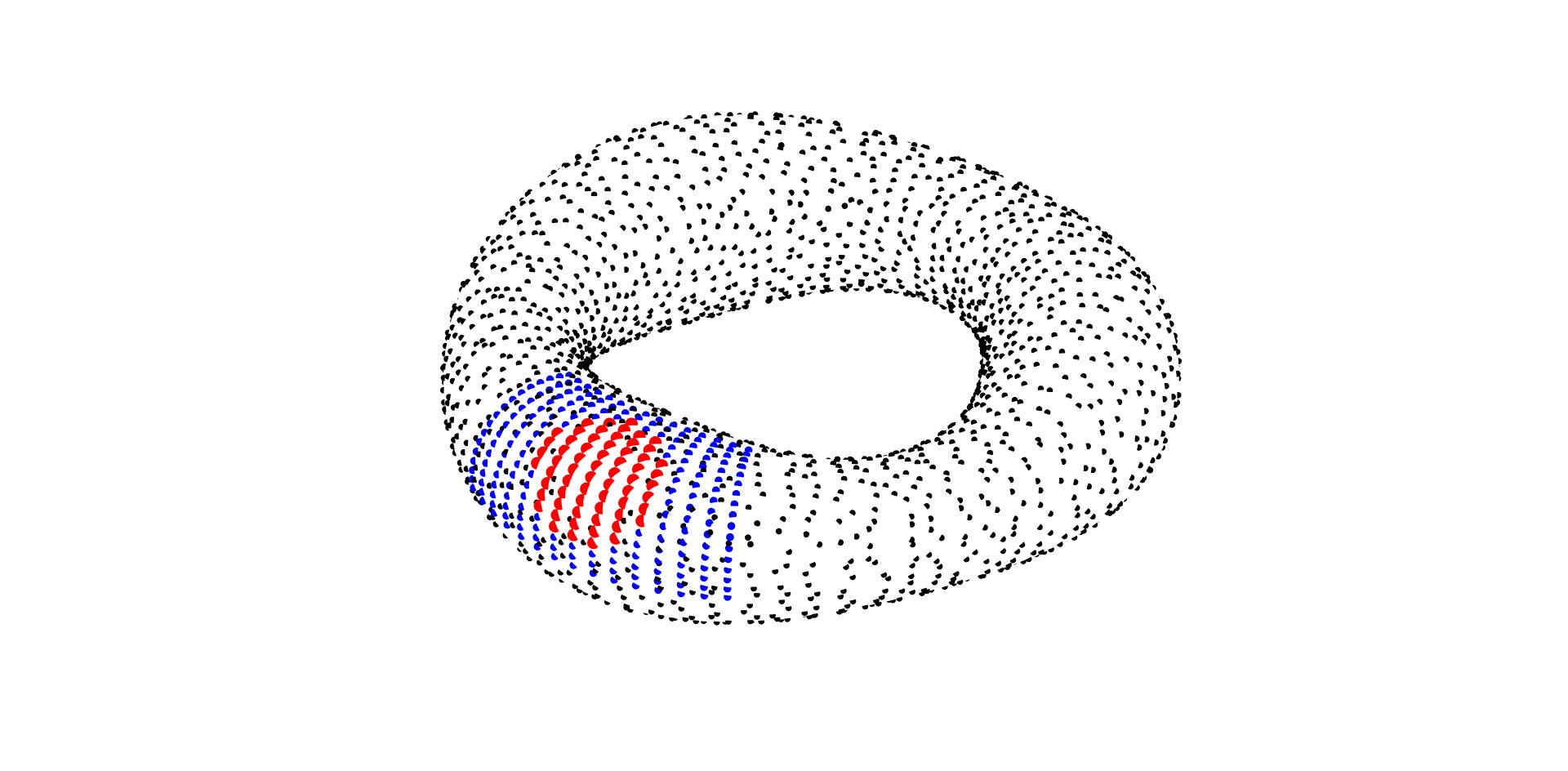}
\caption{Projecting and filling the hole on the boundary of the torus}
\label{RecTorus}
\end{figure}
\medskip

{\bf Illustration of the algorithm for a three-dimensional manifold in $\mathbb{R}^4$.}\hfill

\medskip
We consider the three-dimensional manifold in $\mathbb{R}^4$ defined by
\[
x_1^2 + x_2^2 + x_3^2 - x_4^2 = 0.
\]

We assume that the data are given  
at grid points on cross-sections  
of the manifold  
at fixed values of $x_4$. Portions of four such cross-sections  
are displayed on the left side of Figure~\ref{Ballsin4D}.  
The data points coincide  
with the vertices of the meshes.  
A region of missing data is visible at the center of each mesh. A portion $R$ of the associated hyperplane $\Pi$,  
as described in the hole-filling algorithm,  
is shown on the right side of the figure.  
Its natural position is near the missing region.  
However, for clarity, its center is shifted 
to the red point.

The next steps of the algorithm  
are to overlay a mesh on $R$,  
apply the MMLS projection,  
and then perform functional hole filling  
for each of the four components  
of the projected data.

The data structure in this example  
is selected to facilitate  
a clear presentation  
of the data,  
the missing region,  
and the reference hyperplane  
for a three-dimensional manifold.  
Under these circumstances,  
the hole-filling problem  
is significantly simplified.  
It can be treated directly  
on the given grid structure,  
separately on each cross-section,  
using a hole-filling algorithm  
for functions defined on two-dimensional grids.

\begin{figure}
\centering
\includegraphics[width=4 in]{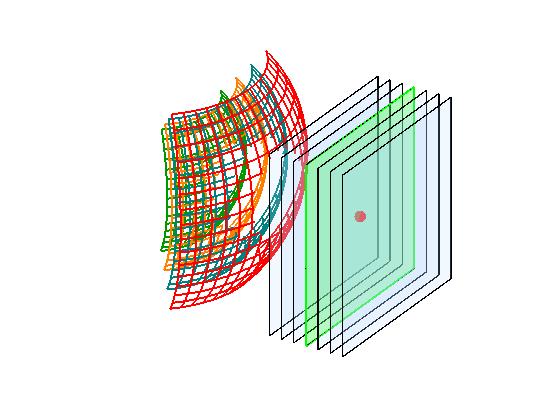}
\caption{Cross-sectional data containing a missing region and the associated reference hyperplane}
\label{Ballsin4D}
\end{figure}
\medskip

\medskip

\appendix

\section{A bound on the inverse of high-order difference operators}

\begin{proposition}
Let $A\in\mathbb R^{n\times n}$ be the matrix
of the centered $2k$-th difference operator
with homogeneous exterior conditions.
Then
\[
\|A^{-1}\|_\infty
\le
\frac{|E_{2k}|}{2^{2k}(2k)!}\,(n+1)^{2k},
\]
and in particular
\[
\|A^{-1}\|_\infty
\le
\Bigl(\frac{(n+1)^2}{8}\Bigr)^k,
\qquad k\ge2.
\]
\end{proposition}

\begin{proof}
Write $A=(-1)^kB^k$, where $B$ is the Dirichlet second-difference matrix.
Hence $\|A^{-1}\|_\infty=\|B^{-k}\|_\infty$.
Since $B$ is an $M$-matrix, $B^{-k}\ge0$ entrywise, and therefore
\[
\|B^{-k}\|_\infty=\|B^{-k}\mathbf1\|_\infty.
\]
Let $u=B^{-k}\mathbf1$.
Then $B^ku=\mathbf1$.
Comparing this discrete problem with the continuous boundary-value problem
\[
(-1)^kU^{(2k)}=1\quad\text{on }(0,1),
\qquad
U^{(2r)}(0)=U^{(2r)}(1)=0,\ r=0,\dots,k-1,
\]
yields
\[
u_i\le (n+1)^{2k}U(x_i),\qquad x_i=\frac{i}{n+1}.
\]
Thus
\[
\|A^{-1}\|_\infty
\le
(n+1)^{2k}\max_{x\in[0,1]}U(x).
\]
Using the sine-series representation of $U$, one finds
\[
\max_{x\in[0,1]}U(x)
=
\frac{4}{\pi^{2k+1}}\beta(2k+1)
=
\frac{|E_{2k}|}{2^{2k}(2k)!}.
\]
Hence
\[
\|A^{-1}\|_\infty
\le
\frac{|E_{2k}|}{2^{2k}(2k)!}(n+1)^{2k},
\]
and therefore
\[
\|A^{-1}\|_\infty
\le
\Bigl(\frac{(n+1)^2}{8}\Bigr)^k.
\]
\end{proof}

\end{document}